\documentclass{article}

\usepackage[utf8]{inputenc}
\usepackage{amsmath}
\usepackage{amsthm}
\usepackage{amssymb}
\usepackage{amscd}
\usepackage[all]{xy} 
\usepackage{latexsym}
\usepackage{extpfeil}
\usepackage{graphicx}
\usepackage{changebar,color}
\usepackage{tikz-cd}
\usepackage{bbold}
\usepackage[T1]{fontenc}
\usepackage[english]{babel}
\usepackage{tabularx}
\usepackage{comment}
\usepackage{enumitem}

\newtheorem{definition}{Definition}[section]

\newtheorem{lemma}[definition]{Lemma}

\newtheorem{theorem}[definition]{Theorem}

\theoremstyle{definition}\newtheorem{rema}[definition]{Remark}

\def\Hom{\mathrm{Hom}}
\def\id{\mathrm{id}}

\def\ch{\mathrm{ch}}
\def\KL{\mathrm{KL}}

\def\cC{\mathcal{C}}

\def\cD{\mathcal{D}}

\def\cY{\mathcal{Y}}

\title{The Kazhdan-Lusztig category of $\mathfrak{osp}_{1|2n}$ at irrational levels}
\author{Thomas Creutzig\footnote{Department Mathematik, Friedrich-Alexander Universit\"at Erlangen-N\"urnberg, 
91058 Erlangen, Germany\\
\textit{E-mail}: \texttt{thomas.creutzig@fau.de}},\, Robert McRae\footnote{Yau Mathematical Sciences Center, Tsinghua University, Beijing 100084, China\\
\textit{E-mail}: \texttt{rhmcrae@tsinghua.edu.cn}}\,\, and Jinwei Yang\footnote{School of Mathematical Sciences, Shanghai Jiao Tong University, Shanghai 200240, China\\
\textit{E-mail}: \texttt{jinwei2@sjtu.edu.cn}}}
\date{}

\begin{document}

\maketitle

\begin{abstract}
    We prove the Kazhdan-Lusztig correspondence for the  Lie superalgebra $\mathfrak{osp}_{1|2n}$ at irrational levels, that is, we show the category $\KL_k^{\rm ev}(\mathfrak{osp}_{1|2n})$ of finite-length even ordinary modules for the affine vertex operator superalgebra of $\mathfrak{osp}_{1|2n}$ at level $k \in \mathbb{C} \setminus \mathbb{Q}$ is braided tensor equivalent to the category of finite-dimensional even weight modules for the quantum group of $\mathfrak{osp}_{1|2n}$ at parameter $q = e^{\pi i/(2k+2n+1)}$. We also prove that $\KL_k^{\rm ev}(\mathfrak{osp}_{1|2n})$ is braided tensor equivalent to the category $\KL_\ell^{\rm ns}(\mathfrak{so}_{2n+1})$ of finite-length ordinary modules with non-spinorial top level for the affine vertex operator algebra of $\mathfrak{so}_{2n+1}$ at level $\ell$ such that
    \[
     \frac{1}{\ell+ 2n-1} = \frac{1}{2k+2n+1} + 1 \ \ ({\rm mod}\  2\mathbb Z).
    \]
Consequently, by gluing vertex operator (super)algebras via tensor categories, we construct a few new families of simple conformal vertex (super)algebras, including the mixed kernel VOAs that were the missing ingredient for proving certain Feigin-Frenkel type dualities in previous work of the first-named author with Linshaw, Nakatsuka, and Sato.
\end{abstract}

\tableofcontents

\section{Introduction}

Representations of vertex operator algebras (VOAs) and of quantum groups are two natural sources of braided tensor categories. Since understanding VOA representation theory can be difficult, one major objective is to establish braided tensor equivalences with representation categories of quantum groups. Such equivalences are called Kazhdan-Lusztig correspondences after the pioneering work \cite{KL1, KL2, KL3, KL4} for the category of finite-length ordinary modules for the affine Lie algebra of a simple Lie algebra at a shifted level $k+h^\vee\notin\mathbb{Q}_{\geq 0}$. One now wants to generalize in various directions: 
\begin{enumerate}
    \item To include positive rational shifted levels, so far done for $\widehat{\mathfrak{sl}}_2$ \cite{MY}.
    \item To modules for affine Lie superalgebras, so far done for $\widehat{\mathfrak{gl}}_{1|1}$ \cite{CLR, CN}.
    \item Beyond the category of ordinary modules; for example, there now is a correspondence between weight modules of $\widehat{\mathfrak{sl}}_2$ at an admissible level with the quantum supergroup of type $\mathfrak{sl}_{2|1}$ \cite{CL}.
    \item To modules for affine $W$-algebras, so far done for ordinary modules for $W$-algebras of simply-laced simple Lie algebras at irrational levels \cite{CDN}.
    \item Beyond affine VOAs and $W$-algebras; so far there are correspondences for the singlet and triplet VOAs with (unrolled) small quantum groups of $\mathfrak{sl}_2$ \cite{GN, CLR2, CLR}.
\end{enumerate}  

These generalizations typically involve non-semisimple braided tensor categories, but there is one family of simple Lie superalgebras, of type $\mathfrak{osp}_{1|2n}$, that behave much like simple Lie algebras. Especially, every finite-dimensional $\mathfrak{osp}_{1|2n}$-module is completely reducible. In fact, forgetting parity, the category of finite-dimensional $\mathfrak{osp}_{1|2n}$-modules is equivalent to the category of finite-dimensional non-spinorial $\mathfrak{so}_{2n+1}$-modules \cite{Rittenberg:1981fm}.
The purpose of this note is to observe that this equivalence persists in the affine setting. In particular, we establish the Kazhdan-Lusztig correspondence for $\widehat{\mathfrak{osp}}_{1|2n}$ at any irrational level. 

\subsection{The main result}

For $\mathfrak{g}$ a simple Lie algebra or $\mathfrak{osp}_{1\vert 2n}$, let $\KL_k(\mathfrak{g})$ denote the Kazhdan-Lusztig category of finite-length ordinary modules for the universal affine vertex operator (super)algebra $V_k(\mathfrak{g})$. If $k\in\mathbb{C}\setminus\mathbb{Q}$, then $\KL_k(\mathfrak{g})$ is semisimple with simple objects given by the Weyl modules induced from finite-dimensional simple $\mathfrak{g}$-modules. Thus the simple objects of $\KL_k(\mathfrak{g})$ are indexed by the dominant integral weights of $\mathfrak{g}$. Similarly, the category $\mathcal{C}_q(\mathfrak{g})$ of finite-dimensional weight modules for the quantum group $U_q(\mathfrak{g})$, with $q$ not a root of unity, is semisimple with simple objects labeled by dominant integral weights of $\mathfrak{g}$. 

In this paper, we take $\mathfrak{g}=\mathfrak{so}_{2n+1},\mathfrak{osp}_{1\vert 2n}$. The set of dominant integral weights $P^+$ of $\mathfrak{osp}_{1\vert 2n}$ can be identified with the set of non-spin dominant integral weights of $\mathfrak{so}_{2n+1}$. Also, $P^+$ has a natural $\mathbb{Z}/2\mathbb{Z}$-grading, so we say a simple module in $\KL_k(\mathfrak{osp}_{1\vert 2n})$ or $\mathcal{C}_q(\mathfrak{osp}_{1\vert 2n})$ is even if its highest weight vector has the same parity as its highest weight $\Lambda\in P^+$. Then our main result is the following braided tensor equivalences:
\[
\begin{tikzcd}
\KL_\ell^{\rm ns}(\mathfrak{so}_{2n+1}) \arrow[r, leftrightarrow, "\cong"] \arrow[d, leftrightarrow, "\cong"']
  & \mathcal{C}_{\zeta}^{\rm ns}(\mathfrak{so}_{2n+1}) \\
\KL_k^{\rm ev}(\mathfrak{osp}_{1|2n}) \arrow[r, leftrightarrow, "\cong"'] 
  & \mathcal{C}_{q}^{\rm ev}(\mathfrak{osp}_{1|2n}) \arrow[u, leftrightarrow, "\cong"']
\end{tikzcd}
\]
where ns and ev indicate the subcategories of non-spin and even representations. More precisely, we prove (see Theorems \ref{thm:ss-BTC} and \ref{thm:all} below):
\begin{theorem} \label{thm:intro}
If $k,\ell\in\mathbb{C}\setminus\mathbb{Q}$, then $\KL_k^{\rm ev}(\mathfrak{osp}_{1\vert 2n})$ and $\KL_\ell^{\rm ns}(\mathfrak{so}_{2n+1})$ are rigid braided tensor categories. Moreover, if
\[
\frac{1}{2k+2n+1} = \frac{1}{\ell+ 2n-1} + 1 \ \ ({\rm mod}\  2\mathbb Z)
\]
and we define $q,\zeta\in\mathbb{C}^\times$ by $q=e^{\pi i/(2k+2n+1)}$ and $\zeta=e^{\pi i/2(\ell+2n-1)}$  (so $q=-\zeta^2$),
then there are braided tensor equivalences
\begin{equation*}
\KL_k^{\rm ev}(\mathfrak{osp}_{1|2n}) \cong\mathcal{C}_{q}^{\rm ev}(\mathfrak{osp}_{1|2n}) \cong \mathcal{C}_{\zeta}^{\rm ns}({\mathfrak{so}_{2n+1}})\cong  \KL_\ell^{\rm ns}(\mathfrak{so}_{2n+1})
\end{equation*}
which match simple objects labeled by $\Lambda\in P^+$ in each category.
\end{theorem}

Our main novel contribution is the first equivalence in this theorem, which is a Kazhdan-Lusztig correspondence for $\mathfrak{osp}_{1|2n}$ at irrational levels. We prove it using Tuba-Wenzl's classification of semisimple rigid braided tensor categories with non-spinorial $\mathfrak{so}_{2n+1}$ fusion rules \cite{Tuba-Wenzl}.
The middle tensor equivalence was already observed by Zhang in \cite{Zhang:1992b}, and the third is of course part of Kazhdan-Lusztig's work, though we point out that certain technical results in \cite{KL1, KL2, KL3, KL4} were proved under the assumption that $\mathfrak{g}$ is simply laced, so strictly speaking, applying them to $\mathfrak{g}=\mathfrak{so}_{2n+1}$ requires a generalization.


\subsection{Application to mixed kernel VOAs}

Theorem \ref{thm:intro} has an important application to a natural generalization of Feigin-Frenkel duality. Let $W^k(\mathfrak g, f)$ be the universal affine $W$-algebra at level $k$ associated to the simple Lie superalgebra $\mathfrak{g}$ and the even nilpotent element $f$, and let $f_{\text{prin}}$ be the principal nilpotent element. Then Feigin-Frenkel duality \cite{FF} is the isomorphism 
\[
W^k(\mathfrak g, f_{\text{prin}}) \cong W^\ell({}^L\mathfrak g, f_{\text{prin}}), \qquad r^\vee(k+h^\vee_{\mathfrak g})(\ell+h^\vee_{{}^L\mathfrak g}) =1,
\]
where $h^\vee_{\mathfrak g}$ is the dual Coxeter number of $\mathfrak g$, ${}^L\mathfrak g$ is the Lie algebra whose root system is the coroot system of $\mathfrak g$,  and $r^\vee$ is the lacing number of $\mathfrak g$. 

Feigin-Frenkel duality extends to many non-principal $W$-algebras \cite{CL2, CL1, CKLSS}, with two notable differences: the dual vertex operator superalgebra is typically a $W$-superalgebra, and the duality is not an isomorphism of VOAs but only of coset subalgebras. However, for so-called hook-type 
$W$-superalgebras \cite{CL2, CL1}, there is a convolution operation that directly sends a $W$-algebra to its dual $W$-superalgebra, as well as a dual convolution operation sending the $W$-superalgebra back to its dual 
$W$-algebra \cite{CLNS}. The convolution operation is useful because it also acts on modules, which is expected give block-wise equivalences of categories in general. So far, this has been worked out for the subregular $W$-algebras in types $A$ and $B$ \cite{CGNS}.

The convolution operation is a relative semi-infinite Lie algebra cohomology against a certain VOA whose existence was first conjectured in \cite{CG}, where they were called quantum geometric Langlands kernel VOAs due to their appearance in the context of $S$-duality. The first examples were constructed in \cite{CG, CGL}, and then many more by Moriwaki in \cite{Mor}. But there is one convolution operation that mixes $\mathfrak{so}_{2n+1}$ and $\mathfrak{osp}_{1|2n}$, and this requires the existence of mixed kernel VOAs (see \cite[Section 2]{CLNS}). 

In Theorem \ref{thm:mixed-kernel-VOAs} of this paper, we use Theorem \ref{thm:intro} and a superalgebra generalization of the gluing construction in \cite{CKM2} to prove the existence of mixed kernel VOAs (which, strictly speaking, are conformal vertex superalgebras in general):
Identifying the dominant integral weights $P^+$ of $\mathfrak{osp}_{1\vert 2n}$ with the non-spinorial ones of $\mathfrak{so}_{2n+1}$,
let $M_k(\Lambda)$ and $N_k(\Lambda)$ be the Weyl modules of $V_k(\mathfrak{osp}_{1|2n})$ and $V_k(\mathfrak{so}_{2n+1})$ of highest weight $\Lambda\in P^+$ and level $k\in\mathbb{C}\setminus\mathbb{Q}$.
Our result is that the following have simple conformal vertex (super)algebra structure:
\begin{equation*}
    \begin{split}
        A[\mathfrak{osp}_{1|2n}, \mathfrak{osp}_{1|2n}, m, k] &= \bigoplus_{\Lambda \in P^+} M_k(\Lambda) \otimes M_{k'}(\Lambda) \\ 
        A[\mathfrak{osp}_{1|2n}, \mathfrak{so}_{2n+1}, p, k] &= \bigoplus_{\Lambda \in P^+} M_k(\Lambda) \otimes N_{\ell}(\Lambda) \\ 
        A[\mathfrak{so}_{2n+1}, \mathfrak{so}_{2n+1}, m, \ell]^{\rm ns} &= \bigoplus_{\Lambda \in P^+} N_\ell(\Lambda) \otimes N_{\ell'}(\Lambda) \\ 
    \end{split}
\end{equation*}
for $k, \ell\in\mathbb{C}\setminus\mathbb{Q}$ and $m\in 2\mathbb{Z}$, $p\in 2\mathbb{Z}+1$, where $k'$, $\ell'$ are related to $k$, $\ell$ by
\[
\frac{1}{2k+2n+1} + \frac{1}{2k'+2n+1} = m =  \frac{1}{\ell+2n-1} + \frac{1}{\ell'+2n-1} 
\]
in the first and third cases, and $\ell$ is related to $k$ in the second case by
\[
\frac{1}{2k+2n+1} + \frac{1}{\ell+2n-1} = p.
\]
In particular, $A[\mathfrak{osp}_{1|2n}, \mathfrak{so}_{2n+1}, p, k]$ gives the mixed kernel VOAs, and the cases $p = \pm 1$ are the ones relevant for the Feigin-Frenkel-type duality via convolution.

\section{The Lie superalgebras $\mathfrak{osp}_{1|2n}$ and $\mathfrak{so}_{2n+1}$ and their representations}\label{finite}

We start from the correspondence between finite-dimensional representations of $\mathfrak{osp}_{1|2n}$ and $\mathfrak{so}_{2n+1}$. Consider the lattice $Q=\mathbb Z^n$ with orthonormal basis $\{ \epsilon_1, \dots, \epsilon_n\}$. Then the root systems are
\begin{equation}
    \begin{split}
        &\Delta(\mathfrak{osp}_{1|2n})= \Delta(\mathfrak{osp}_{1|2n})_{\bar{0}} \cup \Delta(\mathfrak{osp}_{1|2n})_{\bar{1}}, \\
        &\Delta(\mathfrak{osp}_{1|2n})_{\bar{0}} = \{ \pm \epsilon_i \pm \epsilon_j \mid 1 \leq i, j \leq n \}, \\
        &\Delta(\mathfrak{osp}_{1|2n})_{\bar{1}} = \{ \pm \epsilon_i \mid 1 \leq i \leq n \}, \\
        &\Delta(\mathfrak{so}_{2n+1}) = \{ \pm \epsilon_i \pm \epsilon_j \mid 1 \leq i < j \leq n \}  \cup 
        \{ \pm \epsilon_i \mid 1 \leq i \leq n \}. 
    \end{split}
\end{equation}
That is, the short, respectively long, roots of $\mathfrak{so}_{2n+1}$ are identified with the odd, respectively short even, roots
of $\mathfrak{osp}_{1|2n}$.
In particular,
\[
\Pi = \{ \alpha_1 = \epsilon_1 - \epsilon_2, \dots, \alpha_{n-1}= \epsilon_{n-1}-\epsilon_n, \alpha_n =\epsilon_n\}
\]
is a system of simple roots for both $\mathfrak{so}_{2n+1}$ and $\mathfrak{osp}_{1\vert 2n}$. 
The fundamental weights are
\[
\omega_i = \epsilon_1 + \dots + \epsilon_i, \qquad 1 \leq i  \leq n-1, \qquad
\omega_n = \frac{1}{2}(\epsilon_1 + \dots + \epsilon_n).
\]
The Weyl vector of $\mathfrak{so}_{2n+1}$ is the half sum of all positive roots, while that of $\mathfrak{osp}_{1|2n}$ is the difference of the half sums of positive even roots and positive odd roots. These two vectors coincide and equal
\[
\rho = \omega_1 + \dots + \omega_n.
\]
Let 
\[
P^+ = \mathbb Z_{\geq 0}\omega_1 + \dots + \mathbb Z_{\geq 0}\omega_{n-1} + \mathbb Z_{\geq 0}2\omega_n.
\]
This is the set of highest weights of finite-dimensional irreducible representations of 
$\mathfrak{osp}_{1|2n}$ and of non-spinorial irreducible representations of $\mathfrak{so}_{2n+1}$. 

Since $\mathfrak{osp}_{1|2n}$ is a superalgebra, the category of finite-dimensional $\mathfrak{osp}_{1|2n}$-modules is a supercategory and we need to distinguish between a module and its parity reversal. Define the parity of a weight by the following additive map
\begin{align*}
p: Q &\longrightarrow \mathbb{Z}/2\mathbb{Z}\\
\sum_{i=1}^n m_i\epsilon_i & \longmapsto \sum_{i=1}^n m_i
\pmod{2}.
\end{align*}
Note that this parity grading of $Q$ is consistent with the division of the roots of $\mathfrak{osp}_{1\vert 2n}$ into even and odd roots.
We say a weight $\mathfrak{osp}_{1|2n}$-module $M$ is even if
\[
M_\mu\subseteq M_{p(\mu)}
\]
for every weight $\mu$, that is, each weight space of $M$ is parity homogeneous with the same parity as its weight. 

Consider the underlying category of the supercategory of finite-dimensional $\mathfrak{osp}_{1|2n}$-modules, which has the same objects as the supercategory but retains only the even morphisms. Its full subcategory consisting of even modules is closed under the tensor product. Indeed, if $M$ and $N$ are even modules, then for homogeneous weight vectors $m\in M_\mu,
\; n\in N_\nu$,
the tensor product $m\otimes n$ has weight
\[
\operatorname{wt}(m\otimes n)=\mu+\nu
\]
and parity
\[
p(m\otimes n) = p(m)+p(n) = p(\mu)+p(\nu) = p(\mu+\nu)
\]
since $p$ is additive.
Therefore,
\[
(M\otimes N)_\gamma
\subseteq
(M\otimes N)_{p(\gamma)}
\]
for every weight $\gamma$, showing that $M\otimes N$ is also an even module.

For a module $M$ with weight space decomposition $M = \bigoplus_\mu M_\mu$, we call
\[
\ch[M] = \sum_\mu \text{dim}(M_\mu) e^\mu \  \in \mathbb Z[Q] 
\]
the {\em graded character}.
\begin{theorem} \textup{(Rittenberg-Scheunert \cite{{Rittenberg:1981fm}})}
For $\Lambda \in P^+$, let $M(\Lambda)$ be the even irreducible highest-weight module of 
    $\mathfrak{osp}_{1|2n}$ of highest-weight $\Lambda$, and let $N(\Lambda)$ be the corresponding one of $\mathfrak{so}_{2n+1}$. Then the graded characters coincide:
    \[
    \ch[M(\Lambda)] = \ch[N(\Lambda)],
    \]
    and the tensor rings coincide:
    \[
    M({\Lambda}) \otimes M({\Lambda'}) \cong \bigoplus_{\Lambda'' \in P^+} N_{\Lambda, \Lambda'}^{\ \ \Lambda''}\  M({\Lambda''}), \quad 
    N({\Lambda}) \otimes N({\Lambda'}) \cong \bigoplus_{\Lambda'' \in P^+} N_{\Lambda, \Lambda'}^{\ \ \Lambda''}  \ N({\Lambda''}), 
    \]
    where the tensor product multiplicities are given by
\begin{align*}
N_{\Lambda, \Lambda'}^{\ \ \Lambda''} & = \dim\Hom(M({\Lambda}) \otimes M({\Lambda'}), M({\Lambda''}))\\
&= \dim\Hom(N({\Lambda}) \otimes N({\Lambda'}), N({\Lambda''})).
\end{align*}
\end{theorem}

This theorem implies that the category of finite-dimensional non-spinorial representations of $\mathfrak{so}_{2n+1}$ is equivalent as a tensor category to the full subcategory of finite-dimensional even representations of $\mathfrak{osp}_{1|2n}$ via the identification $N(\Lambda) \mapsto M(\Lambda)$. To see why, note that the forgetful functors from both categories to the category of finite-dimensional vector spaces induce tensor equivalences with their images, and the images of these two forgetful functors coincide.


\section{The affine Lie superalgebras $\widehat{\mathfrak{osp}}_{1|2n}$ and $\widehat{\mathfrak{so}}_{2n+1}$ and their representations}\label{aff}

We now consider representations of the affinizations of $\mathfrak{osp}_{1\vert 2n}$ and $\mathfrak{so}_{2n+1}$ at irrational levels, $k \in \mathbb C\setminus \mathbb Q$. We denote by $M_k(\Lambda)$ and $N_k(\Lambda)$ the Weyl $\widehat{\mathfrak{osp}}_{1|2n}$- and $\widehat{\mathfrak{so}}_{2n+1}$-modules induced from $M(\Lambda)$ and $N(\Lambda)$, respectively. 
The conformal weights of their top levels, coming from the usual Sugawara constructions, are
\[
h(M_k(\Lambda)) = \frac{(\Lambda| \Lambda+2\rho)_\mathfrak{osp}}{2(k+n+\frac{1}{2})}, 
\qquad h(N_k(\Lambda)) = \frac{(\Lambda| \Lambda+2\rho)_\mathfrak{so}}{2(k+2n-1)}.
\]
The bilinear forms are normalized so that long roots have square norm two, so
\[
(\Lambda| \Lambda+2\rho)_\mathfrak{osp} = \frac{1}{2}(\Lambda| \Lambda+2\rho)_\mathfrak{so}.
\]
Thus 
\[
h(M_k(\Lambda)) = \frac{(\Lambda| \Lambda+2\rho)_\mathfrak{so}}{2(2k+2n+1)}.
\]
We define the parity of a weight of $\widehat{\mathfrak{osp}}_{1|2n}$ to be the parity of its restriction to a weight of $\mathfrak{osp}_{1|2n}$, and we say a module for the universal  affine vertex operator superalgebra $V_k(\mathfrak{osp}_{1|2n})$ is even if the parity of any of its weight vectors agrees with the parity of the corresponding weight. 

Define the Kazhdan-Lusztig category $\KL_k(\mathfrak{osp}_{1\vert 2n})$ to be the category whose objects are finitely-generated grading-restricted generalized $V_k(\mathfrak{osp}_{1\vert 2n})$-modules (in particular, every module has finite-dimensional conformal weight spaces and a lower bound on conformal weights) and whose morphisms are even $V_k(\mathfrak{osp}_{1\vert 2n})$-module homomorphisms (equivalently, even $\widehat{\mathfrak{osp}}_{1\vert 2n}$-module homomorphisms). For $k\in\mathbb{C}\setminus\mathbb{Q}$, it is shown in \cite[Theorems 5.10 and 5.11]{AALY} that $\KL_k(\mathfrak{osp}_{1\vert 2n})$ is a semisimple braided tensor category, with simple objects $M_k(\Lambda)$ for $\Lambda\in P^+$ and their parity reversals.
Let $\KL_k^{\rm ev}(\mathfrak{osp}_{1|2n})$ be the full subcategory of $\KL_k(\mathfrak{osp}_{1\vert 2n})$ consisting of even $V_k(\mathfrak{osp}_{1|2n})$-modules; it is semisimple with simple objects $M_k(\Lambda)$ for $\Lambda\in P^+$.

\begin{theorem}\label{thm:ss-BTC}
    $\KL_k^{\rm ev}(\mathfrak{osp}_{1|2n})$ is a semisimple braided ribbon tensor category.
\end{theorem}


\begin{proof}
Since $\KL^{\mathrm{ev}}_k(\mathfrak{osp}_{1\vert 2n})$ is a full subcategory of the semisimple braided tensor category $\KL_k(\mathfrak{osp}_{1\vert 2n})$, it is semisimple. Since $\KL^{\mathrm{ev}}_k(\mathfrak{osp}_{1\vert 2n})$ also contains the tensor unit $M_k(0)=V_k(\mathfrak{osp}_{1\vert 2n})$, it will be a braided tensor category if it is closed under the tensor product $\boxtimes$ on $\KL_k(\mathfrak{osp}_{1\vert 2n})$. To show this,
because $\KL_k(\mathfrak{osp}_{1\vert 2n})$ is semisimple, it is enough to determine the multiplicity of every simple object $W$ of $\KL_k(\mathfrak{osp}_{1\vert 2n})$ in $M_k(\Lambda)\boxtimes M_k(\Lambda')$ for $\Lambda,\Lambda'\in P^+$. By the Frenkel-Zhu-Li fusion rules theorem \cite{FZ, Li} (see also \cite[Theorem 3.1]{CMY3}), this multiplicity is the same as the dimension of the space of even $\mathfrak{osp}_{1\vert 2n}$-homomorphisms from $M(\Lambda)\otimes M(\Lambda')$ to the top level of $W$. It follows that
\[
M_k({\Lambda}) \boxtimes M_k({\Lambda'}) \cong \bigoplus_{\Lambda'' \in P^+} N_{\Lambda, \Lambda'}^{\ \ \Lambda''} \ M_k({\Lambda''}), 
\]
which shows $\KL^{\mathrm{ev}}_k(\mathfrak{osp}_{1\vert 2n})$ is closed under $\boxtimes$ and so is a braided tensor category. 

The above fusion rules also imply that every object of $\KL^{\mathrm{ev}}_k(\mathfrak{osp}_{1\vert 2n})$ is non-negligible and has moderate growth in the sense of \cite[Section 1]{EP}, and thus $\KL^{\mathrm{ev}}_k(\mathfrak{osp}_{1\vert 2n})$ is rigid by \cite[Theorem 1.1]{EP}. Finally, since $V_k(\mathfrak{osp}_{1\vert 2n})$ is a $\mathbb{Z}$-graded vertex operator superalgebra, $\KL^{\mathrm{ev}}_k(\mathfrak{osp}_{1\vert 2n})$ also has ribbon twist $e^{2\pi i L_0}$ as usual and thus is a ribbon tensor category.
\end{proof}

Similarly, for $k\in\mathbb{C}\setminus\mathbb{Q}$, let $\KL_k(\mathfrak{so}_{2n+1})$ be the category of finitely-generated grading-restricted generalized modules for the universal affine VOA $V_k(\mathfrak{so}_{2n+1})$. It is a semisimple braided ribbon tensor category by the non-simply-laced generalization of \cite{KL1,KL2,KL3,KL4} (existence of braided tensor structure also follows from \cite{Zhang}, and rigidity from \cite{EP}). Then let $\KL_k^{\rm ns}(\mathfrak{so}_{2n+1})$ be the full subcategory of $\KL_k(\mathfrak{so}_{2n+1})$ of $V_k(\mathfrak{so}_{2n+1})$-modules whose top levels are non-spinorial $\mathfrak{so}_{2n+1}$-modules; it is a semisimple braided ribbon tensor category with simple objects $N_k(\Lambda)$, $\Lambda\in P^+$, just as in the proof of Theorem \ref{thm:ss-BTC}.

\section{Braidings on categories of non-spinorial $\mathfrak{so}_{2n+1}$ type}\label{qg}

We say a semisimple braided tensor category $\mathcal{C}$ is of \textit{non-spinorial $\mathfrak{so}_{2n+1}$ type} if its Grothendieck ring is the same as that of the category of non-spinorial finite-dimensional $\mathfrak{so}_{2n+1}$-modules. Thus the simple objects $X(\Lambda)$ of such a category $\mathcal{C}$ are labeled by $\Lambda \in P^+$ and have fusion rules
\[
X({\Lambda}) \otimes X({\Lambda'}) \cong \bigoplus_{\Lambda'' \in P^+} N_{\Lambda, \Lambda'}^{\ \ \Lambda''}\  X({\Lambda''}). 
\]
So far we have seen that $\KL_k^{\rm ns}(\mathfrak{so}_{2n+1})$ and $\KL_k^{\rm ev}(\mathfrak{ops}_{1\vert 2n})$ for $k\in\mathbb{C}\setminus\mathbb{Q}$ are of non-spinorial $\mathfrak{so}_{2n+1}$ type, and in this section, we will introduce further examples from the quantum (super)groups associated to $\mathfrak{so}_{2n+1}$ and $\mathfrak{osp}_{1\vert 2n}$. 

As mentioned in \cite[Section 7.8]{Tuba-Wenzl}, categories of non-spinorial $\mathfrak{so}_{2n+1}$ type admit more than one braiding, so in this section, we will also determine a choice of braiding on each of our examples.
Any category $\mathcal{C}$ of non-spinorial $\mathfrak{so}_{2n+1}$ type is tensor generated by its simple object $X(\omega_1)$, so any braiding on $\mathcal{C}$ is determined by the self-braiding automorphism $c_{\omega_1,\omega_1}$ of 
\begin{equation*}
X(\omega_1)\otimes X(\omega_1)\cong X(2\omega_1)\oplus X(\omega_2)\oplus X(0),
\end{equation*}
which is in turn determined by its three eigenvalues $a_{2\omega_1}$, $a_{\omega_1}$, and $a_0$ on $X(2\omega_1)$, $X(\omega_2)$, and $X(0)$, respectively. Thus it is sufficient to determine a choice of these eigenvalues for each category of interest.

\subsection{Quantum (super)groups}\label{subsec:qg-braidings}

Let $\mathcal{C}_\zeta(\mathfrak{so}_{2n+1})$ be the semisimple braided tensor category of finite-dimensional weight modules for the quantum group $U_\zeta(\mathfrak{so}_{2n+1})$ for $\zeta$ not a root of unity (see for example \cite[Section 1.3]{BK} or \cite[Section 5.7]{EGNO}). We use $\zeta$ for the quantum parameter instead of $q$ to avoid confusion later with the parameter denoted $q$ in \cite{Tuba-Wenzl} (which is $\zeta^2$). The category $\mathcal{C}_\zeta(\mathfrak{so}_{2n+1})$ has the same Grothendieck ring as the category of finite-dimensional $\mathfrak{so}_{2n+1}$-modules, so its full subcategory $\mathcal{C}_\zeta^{\rm ns}(\mathfrak{so}_{2n+1})$ of non-spinorial representations is of non-spinorial $\mathfrak{so}_{2n+1}$ type. We give $\mathcal{C}_\zeta^{\rm ns}(\mathfrak{so}_{2n+1})$ the braiding coming from the standard universal $R$-matrix indicated in \cite[Equation 1.3.12]{BK}.

To determine the eigenvalues $a_{2\omega_1}$, $a_{\omega_2}$, and $a_0$ of the self-braiding $c_{\omega_1,\omega_1}$, we need the ribbon twists and categorical dimensions of some simple objects in $\mathcal{C}^{\rm ns}_\zeta(\mathfrak{so}_{2n+1})$.
The ribbon twist $\theta\in{\rm Aut}(\id_{\mathcal{C}^{\rm ns}_\zeta(\mathfrak{so}_{2n+1})})$ is given by the scalar
\begin{equation*}
    \theta_{\Lambda} = \zeta^{2(\Lambda \vert \Lambda + 2\rho)_{\mathfrak{so}}},
\end{equation*}
on the irreducible $U_\zeta(\mathfrak{so}_{2n+1})$-module of highest weight $\Lambda \in P^+$ 
(see \cite[Exercise~2.2.6]{BK}, where the bilinear form is scaled so that short roots have square length $2$, giving $2(\cdot\vert\cdot)_{\mathfrak{so}}$). In particular, we can calculate
\begin{align*}
\theta_{\omega_1} = \zeta^{4n},\qquad\theta_{2\omega_1} = \zeta^{8n+4},\qquad\theta_{\omega_2} = \zeta^{8n-4},
\end{align*}
and obviously $\theta_0=1$.
The categorical dimension of the irreducible $U_\zeta(\mathfrak{so}_{2n+1})$-module of highest weight $\Lambda\in P^+$ is given by
\begin{equation*}
    d_{\Lambda} = \prod_{\alpha \in \Delta(\mathfrak{so}_{2n+1})_+}\frac{\zeta^{2(\alpha\vert \Lambda+\rho)_{\mathfrak{so}}}-\zeta^{-2(\alpha\vert\Lambda+\rho)_{\mathfrak{so}}}}{\zeta^{2(\alpha\vert \rho)_{\mathfrak{so}}}-\zeta^{-2(\alpha\vert\rho)_{\mathfrak{so}}}}
\end{equation*}
(see \cite[Equation~(3.3.5)]{BK}, where again $2(\cdot\vert \cdot)_{\mathfrak{so}}$ is used). In particular,
\begin{align*}
   & d_{\omega_1} = \frac{(\zeta^{2n+1} -\zeta^{-2n-1)})(\zeta^{2n-1}+\zeta^{-2n+1)})}{\zeta^2-\zeta^{-2}}, \nonumber \\
   & d_{2\omega_1} = \frac{(\zeta^{2n+3} - \zeta^{-2n-3)})(\zeta^{2n-1}+\zeta^{-2n+1)})(\zeta^{4n}-\zeta^{-4n})}{(\zeta^4-\zeta^{-4})(\zeta^2-\zeta^{-2})},\nonumber \\
   & d_{\omega_2} = \frac{(\zeta^{2n+1} - \zeta^{-2n-1})(\zeta^{2n-3}+\zeta^{-2n+3)})(\zeta^{4n}-\zeta^{-4n})}{(\zeta^4-\zeta^{-4})(\zeta^2-\zeta^{-2})},
\end{align*}
and again obviously $d_0=1$.

Now the balancing equation for the braiding and twist says that
\begin{equation*}
c_{Y,X}\circ c_{X,Y}=\theta_{X\otimes Y}\circ(\theta_X^{-1}\otimes\theta_Y^{-1})
\end{equation*}
for any objects $X$, $Y$ in $\mathcal{C}_\zeta^{\rm ns}(\mathfrak{so}_{2n+1})$. Thus taking $X$ and $Y$ to be the irreducible $U_\zeta(\mathfrak{so}_{2n+1})$-module of highest weight $\omega_1$, we find that the eigenvalues of $c_{\omega_1,\omega_1}$ satisfy $a_\Lambda^2=\theta_\Lambda\theta_{\omega_1}^{-2}$ for $\Lambda =2\omega_1,\omega_2,0$, that is,
\begin{equation}\label{balancing}
a_{2\omega_1}^2 = \zeta^4, \qquad a_{\omega_2}^2 = \zeta^{-4}, \qquad a_0^2 = \zeta^{-8n}.
\end{equation}
Moreover, from \cite[Exercise~8.10.15]{EGNO}, the categorical dimension of the irreducible $U_\zeta(\mathfrak{so}_{2n+1})$-module of highest weight $\omega_1$ satisfies
\begin{equation*}
d_{\omega_1} = \theta_{\omega_1}\cdot{\rm Tr}(c_{\omega_1,\omega_1}^{-1}),
\end{equation*}
which yields the equation
\begin{equation}\label{twistbraiding}
\theta_{\omega_1}^{-1}d_{\omega_1} = a_{2\omega_1}^{-1}d_{2\omega_1} + a_{\omega_2}^{-1}d_{2\omega_2} + a_0^{-1}.
\end{equation}
Now it is straightforward to solve \eqref{balancing} and \eqref{twistbraiding} to obtain,
\begin{equation}\label{quantum-so-braiding}
a_{2\omega_1} = \zeta^2, \qquad a_{\omega_2} = -\zeta^{-2},\qquad a_0= \zeta^{-4n}.
\end{equation}
This determines the braiding on $\mathcal{C}^{\rm ns}_\zeta(\mathfrak{so}_{2n+1})$.


Now let $U_q(\mathfrak{osp}_{1|2n})$ be the quantum supergroup of $\mathfrak{osp}_{1|2n}$ for $q$ not a root of unity (see the definitions in \cite[Section~2]{Zhang:1992b} or \cite[Section 3]{B}). For both $U_q(\mathfrak{osp}_{1\vert 2n})$ and $U_\zeta(\mathfrak{so}_{2n+1})$, we can restrict to considering only finite-dimensional weight modules on which the Cartan generators $h_i$ have integer eigenvalues. If we do so, then we get certain quotient algebras of these quantum (super)groups, and by \cite[Theorem~3.3 and Lemma~3.4]{Zhang:1992b} (see also \cite[Theorem~1.1]{Xu-Zhang}), this quotient of $U_q(\mathfrak{osp}_{1\vert 2n})$ is isomorphic to the corresponding quotient of $U_\zeta(\mathfrak{so}_{2n+1})$ at a suitable $\zeta$ (comparing the quantum $\mathfrak{so}_{2n+1}$ commutation relations in \cite{Zhang:1992b} with those in \cite{BK, EGNO}, we see that $\zeta$ is a square root of $-q$ if we define $U_\zeta(\mathfrak{so}_{2n+1})$ using the generators and relations of \cite{BK,EGNO}). Consequently, by \cite[Theorem~2.1 and Theorem~3.5]{Zhang:1992b}, the supercategories of finite-dimensional weight $U_q(\mathfrak{osp}_{1\vert 2n})$-modules and of finite-dimensional $\mathfrak{osp}_{1|2n}$-modules are equivalent as supercategories.


Moreover, we can define even weight $U_q(\mathfrak{osp}_{1\vert 2n})$-modules just as for $\mathfrak{osp}_{1\vert 2n}$. Let $\mathcal{C}_{q}^{\rm ev}({\mathfrak{osp}_{1|2n}})$ be the category of even finite-dimensional weight $U_q(\mathfrak{osp}_{1\vert 2n})$-modules and even $U_q(\mathfrak{osp}_{1\vert 2n})$-homomorphisms. Then $\mathcal{C}_q^{\rm ev}(\mathfrak{osp}_{1\vert 2n})$ is a semisimple braided tensor category which by \cite[Theorem~3.3 and Lemma~3.4]{Zhang:1992b} again is tensor equivalent to $\mathcal{C}_\zeta^{\rm ns}(\mathfrak{so}_{2n+1})$. In particular, $\mathcal{C}_q^{\rm ev}(\mathfrak{osp}_{1\vert 2n})$ is of non-spinorial type. The braiding on $\mathcal{C}_q^{\rm ev}(\mathfrak{osp}_{1\vert 2n})$ depends on the choice of universal $R$-matrix. One choice was used in \cite[Section~III.C.1]{Zhang:1992a} to calculate the eigenvalues $(a_{2\omega_1},a_{\omega_2},a_0)$ of the self-braiding of the even irreducible $U_q(\mathfrak{osp}_{1\vert 2n})$-module of highest weight $\omega_1$, yielding $(-q^{-1},q,q^{2n})$, while a different choice was used in \cite[Lemma 7.2]{B}, yielding $(-q,q^{-1},q^{-2n})$. We will choose the second braiding, from \cite[Lemma 7.2]{B}, as our braiding on $\mathcal{C}^{\rm ev}_q(\mathfrak{osp}_{1\vert 2n})$.

\subsection{$\KL_k(\mathfrak{so}_{2n+1})$ and $\KL_k(\mathfrak{osp}_{1|2n})$}

Now we determine the braidings on $\KL_k^{\rm ns}(\mathfrak{so}_{2n+1})$ and $\KL_k^{\rm ev}(\mathfrak{osp}_{1|2n})$ for $k \in \mathbb C\setminus \mathbb Q$.
So that the following arguments apply to both categories, we use $L_k(\Lambda)$ for $\Lambda\in P^+$ to stand for $N_k(\Lambda)$ in $\KL_k^{\rm ns}(\mathfrak{so}_{2n+1})$ or $M_k(\Lambda)$ in $\KL_k^{\rm ev}(\mathfrak{osp}_{1\vert 2n})$. Let $f$ denote an isomorphism 
\[
L_k(\omega_1)\boxtimes L_k(\omega_1)\longrightarrow L_k(2\omega_1)\oplus L_k(\omega_2)\oplus L_k(0),
\]
and for $\Lambda=2\omega_1,\omega_2,0$, let ${\rm Pr}_\Lambda: L_k(\omega_1)\boxtimes L_k(\omega_1)\rightarrow L_k(\Lambda)$ denote the composition of $f$ with projection to the direct summand $L_k(\Lambda)$. Thus the eigenvalues of the self-braiding automorphism $c_{\omega_1,\omega_1}$ of $L_k(\omega_1)\boxtimes L_k(\omega_1)$ are defined by
\begin{equation}\label{eqn:braiding1}
{\rm Pr}_\Lambda\circ c_{\omega_1,\omega_1}=a_\Lambda\cdot {\rm Pr}_\Lambda
\end{equation}
for $\Lambda=2\omega_1,\omega_2,0$. We proceed to determine each $a_\Lambda$.

For a module $X$ in $\KL_k^{\rm ns}(\mathfrak{so}_{2n+1})$ or $\KL_k^{\rm ev}(\mathfrak{osp}_{1|2n})$, recall the skew-symmetry operation $\Omega$ on the space of (even) intertwining operators of type $\binom{X}{L_k(\omega_1)\,L_k(\omega_1)}$ from \cite[Equation 3.77]{HLZ2}:
\begin{equation*}
    \Omega(\mathcal{Y})(v_1,x)v_2 = (-1)^{p(v_1)p(v_2)}e^{xL_{-1}}\mathcal{Y}(v_2, e^{\pi i}x)v_1,
\end{equation*}
for $v_1,v_2\in L_k(\omega_1)$, where $p$ is the parity (which is always $0$ in the $\KL^{\rm ns}_k(\mathfrak{so}_{2n+1})$ case). Then the self-braiding $c_{\omega_1,\omega_1}$ is characterized by 
\begin{equation}\label{braiding2}
c_{\omega_1,\omega_1}\circ\mathcal{Y}_\boxtimes =\Omega(\mathcal{Y}_\boxtimes),
\end{equation}
where  $\mathcal{Y}_{\boxtimes}$ is the tensor product intertwining operator of type $\binom{L_k(\omega_1) \boxtimes L_k(\omega_1)}{L_k(\omega_1) \; L_k(\omega_1)}$ (see for example \cite[Equation 3.13]{CKM-exts}). Define $\mathcal{E}_\Lambda ={\rm Pr}_\Lambda\circ\mathcal{Y}_\boxtimes$ for $\Lambda=2\omega_1,\omega_2,0$.
\begin{lemma}\label{lemma:braiding}
We have $\Omega(\mathcal{E}_\Lambda)=a_\Lambda\cdot\mathcal{E}_\Lambda$ for $\Lambda=2\omega_1,\omega_2,0$.
\end{lemma}
\begin{proof}
Using \eqref{eqn:braiding1}, \eqref{braiding2}, and the definitions, we have
\begin{equation*}
\Omega(\mathcal{E}_\Lambda) =\Omega({\rm Pr}_\Lambda\circ\mathcal{Y}_\boxtimes) ={\rm Pr}_\Lambda\circ\Omega(\mathcal{Y}_\boxtimes) ={\rm Pr}_\Lambda\circ c_{\omega_1,\omega_1}\circ\mathcal{Y}_\boxtimes =a_\Lambda\cdot{\rm Pr}_\Lambda\circ\mathcal{Y}_\boxtimes =a_\Lambda\cdot\mathcal{E}_\Lambda.
\end{equation*}
\end{proof}

By Lemma \ref{lemma:braiding}, we need to relate $\mathcal{E}_\Lambda$ and $\Omega(\mathcal{E}_\Lambda)$. Let $\mathfrak{g}$ be either $\mathfrak{osp}_{1\vert 2n}$ or $\mathfrak{so}_{2n+1}$, and let $L(\Lambda)$ denote the $\mathfrak{g}$-module $M(\Lambda)$ or $N(\Lambda)$. Then for $v_1,v_2\in L(\omega_1)$ (which is the top level of $L_k(\omega_1)$), the formal series $\mathcal{E}_\Lambda(v_1,x)v_2$ satisfies
\begin{equation}\label{eqn:E}
    \mathcal{E}_\Lambda(v_1,x)v_2 \in x^{h(L_k(\Lambda))-2h(L_k(\omega_1))}(\pi_\Lambda(v_1\otimes v_2)+ x L_k(\Lambda)[[x]])
\end{equation}
where $h(L_k(\Lambda))$ and $h(L_k(\omega))$ are the lowest conformal weights defined in Section \ref{aff} and $\pi_\Lambda$ is a surjective even $\mathfrak{g}$-module homomorphism (see for example \cite[Proposition~3]{CMY3}). Thus
\begin{align}\label{eqn:Omega-E}
& a_\Lambda  \cdot\mathcal{E}(v_1,x)v_2  = \Omega(\mathcal{E}_\Lambda)(v_1,x)v_2= (-1)^{p(v_1)p(v_2)} e^{xL_{-1}}\mathcal{E}_\Lambda(v_2,e^{\pi i}x)v_1\nonumber\\
& \in (-1)^{p(v_1)p(v_2)} (e^{\pi i}x)^{h(L_k(\Lambda))-2h(L_k(\omega_1))}(\pi_\Lambda(v_2\otimes v_1)+ x L_k(\Lambda)[[x]]).
\end{align}
Since $\pi_\Lambda$ is surjective, there exist parity-homogeneous $v_1,v_2\in L(\omega_1)$ such that $\pi_\Lambda(v_1\otimes v_2)\neq 0$. Comparing \eqref{eqn:E} and \eqref{eqn:Omega-E}, we need to see how $\pi_\Lambda$ is affected by exchanging such $v_1$ and $v_2$.

For $\Lambda=2\omega_1$, we can take $v_1=v_2=v_{\omega_1}$, a highest-weight vector in $L(\omega_1)$. Then $\pi_{2\omega_1}(v_1\otimes v_2)=\pi_{2\omega_1}(v_2\otimes v_1)\neq 0$, and we get
\begin{equation*}
    a_{2\omega_1} = (-1)^{p(v_{\omega_1})} e^{\pi i(h(L_k(2\omega_1))-2h(L_k(\omega_1)))}.
\end{equation*}
Note $p(v_{\omega_1})=1$ in the $\mathfrak{osp}_{1\vert 2n}$ case since $\omega_1$ is an odd weight.

For $\Lambda =\omega_2$, note that $2\omega_1-\omega_2=\epsilon_1-\epsilon_2=\alpha_1$, so the weight-$\omega_2$ subspace of $L(\omega_1)\otimes L(\omega_1)$ is spanned by $f_{\alpha_1}v_{\omega_1}\otimes v_{\omega_1}\pm v_{\omega_1}\otimes f_{\alpha_1}v_{\omega_1}$. Since $\alpha_1$ is an even root in the $\mathfrak{osp}_{1\vert 2n}$ case, taking $+$ yields a vector in $L(2\omega_1)$ while taking $-$ yields a highest-weight vector in $L(\omega_2)$. It follows that
\begin{equation*}
\pi_{\omega_2}(f_{\alpha_1}v_{\omega_1}\otimes v_{\omega_1}) =-\pi_{\omega_2}(v_{\omega_1}\otimes f_{\alpha_1}v_{\omega_1})\neq 0,
\end{equation*}
and thus
\begin{align*}
    a_{\omega_2} & = -(-1)^{p(f_{\alpha_1}v_{\omega_1})p(v_{\omega_1})} e^{\pi i(h(L_k(\omega_2))-2h(L_k(\omega_1)))}\\
    &= -(-1)^{p(v_{\omega_1})} e^{\pi i(h(L_k(\omega_2))-2h(L_k(\omega_1)))}
\end{align*}
since $f_{\alpha_1}$ is even.

For $\Lambda=0$, note that $\pi_0$ is the unique (up to scale) non-degenerate invariant bilinear form on the vector representation of $\mathfrak{g}$, which is symmetric for $\mathfrak{so}_{2n+1}$ and supersymmetric for $\mathfrak{osp}_{1\vert 2n}$. Thus $\pi_0(v_1\otimes v_2) =(-1)^{p(v_1)p(v_2)}\pi_0(v_2\otimes v_1)$ for all $v_1,v_2\in L(\omega_1)$, which implies that
\begin{equation*}
    a_0=e^{\pi i(h(L_k(0))-2h(L_k(\omega_1)))}= e^{-2\pi i h(L_k(\omega_1))}
\end{equation*}
since $h(L_k(0))=0$.

Now using the conformal weight formulas $h(M_k(\Lambda))=\frac{1}{2(2k+2n+1)}(\Lambda\vert\Lambda+2\rho)_{\mathfrak{so}}$ and $h(N_k(\Lambda))=\frac{1}{2(k+2n-1)}(\Lambda\vert\Lambda+2\rho)_{\mathfrak{so}}$ from Section \ref{aff}, we get
\begin{align*}
    h(M_k(2\omega_1)) & = \frac{2n+1}{2k+2n+1},\qquad h(M_k(\omega_2)) =\frac{2n-1}{2k+2n+1},\\
    h(M_k(\omega_1)) & = \frac{n}{2k+2n+1},
\end{align*}
and
\begin{align*}
    h(N_k(2\omega_1)) & = \frac{2n+1}{k+2n-1}, \qquad h(N_k(\omega_2)) =\frac{2n-1}{k+2n-1},\\
    h(N_k(\omega_1)) & = \frac{n}{k+2n-1},
\end{align*}
Thus for $\KL_k^{\rm ev}(\mathfrak{osp}_{1\vert 2n})$, the eigenvalues of the self-braiding $c_{\omega_1,\omega_1}$ are:
\begin{equation}\label{braidingaffineosp}
    a_{2\omega_1} = -e^{\pi i/(2k+2n+1)},\quad a_{\omega_2} =e^{-\pi i/(2k+2n+1)},\quad a_0=e^{-2\pi i n/(2k+2n+1)},
\end{equation}
and for 
$\KL_k^{\rm ns}(\mathfrak{so}_{2n+1})$, the eigenvalues of the self-braiding $c_{\omega_1,\omega_1}$ are:
\begin{equation}\label{braidingaffineso}
    a_{2\omega_1} = e^{\pi i/(k+2n-1)},\quad a_{\omega_2} =-e^{-\pi i/(k+2n-1)},\quad a_0=e^{-2\pi i n/(k+2n-1)}.
\end{equation}

\section{Kazhdan-Lusztig correspondences}

In \cite{KL1, KL2, KL3, KL4}, Kazhdan and Lusztig showed that for a simply-laced finite-dimensional simple Lie algebra $\mathfrak{g}$ and for (at least most) $k+h^\vee \notin \mathbb{Q}_{\geq 0}$, the category $\KL_k(\mathfrak{g})$ of finite-length $\widehat{\mathfrak{g}}$-modules of level $k$ whose composition factors are grading-restricted modules for the universal affine VOA $V_k(\mathfrak{g})$ is a rigid braided tensor category. Moreover, $\KL_k(\mathfrak{g})$ is braided tensor equivalent to the category of finite-dimensional weight modules for the quantum group $U_q(\mathfrak{g})$ at $q=e^{\pi i/(k+h^\vee)}$.
It is usually stated that this Kazhdan-Lusztig correspondence generalizes to non-simply-laced $\mathfrak{g}$ as well (see for example Section 2.6 and the proof of Proposition 2.8.3 in \cite{Fi}), in which case $q=e^{\pi i/r^\vee(k+h^\vee)}$ where $r^\vee$ is the lacing number of $\mathfrak{g}$, although some results in \cite{KL4} may not have been proved for all levels in non-simply-laced types.


In this section, for $k\in\mathbb{C}\setminus\mathbb{Q}$, we show how Tuba-Wenzl's classification of tensor categories of non-spinorial type \cite{Tuba-Wenzl} recovers the Kazhdan-Lusztig correspondence for the non-spinorial subcategory of $\KL_k(\mathfrak{so}_{2n+1})$, and we also prove the Kazhdan-Lusztig correspondence for $\mathfrak{osp}_{1\vert 2n}$.
Recall that a category $\mathcal{C}$ of non-spinorial type has simple objects $X(\Lambda)$ labeled by $\Lambda\in P^+$, and that the braiding on $\mathcal{C}$ is completely determined by the eigenvalues $(a_{2\omega_1},a_{\omega_2}, a_0)$ of the self-braiding $c_{\omega_1,\omega_1}$ of
\begin{equation*}
X(\omega_1)\otimes X(\omega_1)\cong X(2\omega_1)\oplus X(\omega_2)\oplus X(0).
\end{equation*}
We call $(a_{2\omega_1},a_{\omega_2}, a_0)$ the \textit{braiding eigenvalue triple} of $\mathcal{C}$.

The main theorem of \cite{Tuba-Wenzl} classifies categories of both $\mathfrak{sp}_{2n}$ and non-spinorial $\mathfrak{so}_m$ types.
In particular, by Proposition 8.4 and Theorem 9.4 of \cite{Tuba-Wenzl}, the simple objects of a category of non-spinorial $\mathfrak{so}_{2n+1}$ type can be labeled so that the braiding eigenvalue triple is one of $(q,-q^{-1}, \pm q^{-2n})$ or $(iq,-iq^{-1},\pm i q^{-2n})$ for some $q\in\mathbb{C}^\times$. Moreover, by \cite[Theorem 9.4]{Tuba-Wenzl}, if two categories of non-spinorial $\mathfrak{so}_{2n+1}$ type have the same braiding eigenvalue triple with $q\notin \lbrace \pm 1,\pm i\rbrace$, then they are tensor equivalent. However, it is not clearly stated in \cite{Tuba-Wenzl} whether the tensor equivalence maps simple objects labeled by $\Lambda$ to each other in this case. We address this issue here:

\begin{theorem}\label{maintheoremtubawenzl}
    Let $\mathcal{C}$ and $\widetilde{\mathcal{C}}$ be categories of non-spinorial $\mathfrak{so}_{2n+1}$ type with simple objects $X(\Lambda)$ and $\widetilde{X}(\Lambda)$ for $\Lambda\in P^+$, respectively. If the braiding eigenvalue triples of $\mathcal{C}$ and $\widetilde{\mathcal{C}}$ are the same and the eigenvalues are not roots of unity, then there is a braided tensor equivalence $F:\mathcal{C}\rightarrow\widetilde{\mathcal{C}}$ such that $F(X(\Lambda))\cong \widetilde{X}(\Lambda)$.
\end{theorem}
\begin{proof}
    By \cite[Theorem 9.4]{Tuba-Wenzl}, there is a tensor equivalence $F:\mathcal{C}\rightarrow\widetilde{\mathcal{C}}$ obtained in the proof of \cite[Theorem 9.3]{Tuba-Wenzl}. To construct $F$, the first step is to take the full monoidal subcategories $\mathcal{A}$ and $\widetilde{\mathcal{A}}$ of $\mathcal{C}$ and $\widetilde{\mathcal{C}}$ with objects $\lbrace X(\omega_1)^{\otimes m}\rbrace_{m\in\mathbb{Z}_{\geq 0}}$ and $\lbrace \widetilde{X}(\omega_1)^{\otimes m}\rbrace_{m\in\mathbb{Z}_{\geq 0}}$, respectively. Using \cite[Theorem 8.5]{Tuba-Wenzl}, there is a monoidal equivalence $G: \mathcal{A}\rightarrow\widetilde{\mathcal{A}}$ such that $G(X(\omega_1)^{\otimes m}) = \widetilde{X}(\omega_1)^{\otimes m}$ for all $m$ and $G(c_{\omega_1,\omega_1})$ is the self-braiding $\widetilde{c}_{\omega_1,\omega_1}\in{\rm End}_{\widetilde{\mathcal{C}}}(\widetilde{X}(\omega_1)^{\otimes 2})$. Now, $\lbrace\id, c_{\omega_1,\omega_1}^{\pm 1}\rbrace$ is a basis of $\mathrm{End}_{\mathcal{C}}(X(\omega_1)^{\otimes 2})$ since one can check that the $3\times 3$ matrix formed from the eigenvalue triples of these three morphisms is non-invertible only if $a_{2\omega_1}$ is $\pm i$ or a $(4n-2)$nd root of unity. Thus for $\Lambda=2\omega_1,\omega_2,0$, the projection $\pi_\Lambda\in\mathrm{End}_\mathcal{C}(X(\omega_1)^{\otimes 2})$ with image $X(\Lambda)$ is a linear combination of $\id$ and $c_{\omega_1,\omega_1}^{\pm1}$, and since the braiding eigenvalue triples of $\mathcal{C}$ and $\widetilde{\mathcal{C}}$ are equal, the projection $\widetilde{\pi}_\Lambda\in\mathrm{End}_{\widetilde{\mathcal{C}}}(\widetilde{X}(\omega_1)^{\otimes 2})$ with image $\widetilde{X}(\Lambda)$ is the same linear combination of $\id$ and $\widetilde{c}_{\omega_1,\omega_1}^{\pm1}$. Since $G(c_{\omega_1,\omega_1})=\widetilde{c}_{\omega_1,\omega_1}$, we thus get $G(\pi_\Lambda) = \widetilde{\pi}_\Lambda$ for $\Lambda=2\omega_1,\omega_2,0$.

    Now in the proof of \cite[Theorem 9.3]{Tuba-Wenzl}, the tensor equivalence $F:\mathcal{C}\rightarrow\widetilde{\mathcal{C}}$ is obtained from $G:\mathcal{A}\rightarrow\widetilde{\mathcal{A}}$ using \cite[Theorems 3.4 and 3.5]{Tuba-Wenzl}. From these theorems, $F$ is the composition of tensor equivalences $\mathcal{C}\rightarrow{\rm Ab}\,\mathcal{A}\rightarrow{\rm Ab}\,\widetilde{\mathcal{A}}\rightarrow\widetilde{\mathcal{C}}$, where the middle arrow is induced by $G$ and ${\rm Ab}\,\mathcal{A}$
    is the idempotent completion of the additive category generated by $\mathcal{A}$ (and ${\rm Ab}\,\widetilde{\mathcal{A}}$ is similar).  Objects of ${\rm Ab}\,\mathcal{A}$ are ordered pairs $(X,\pi)$ where $X$ is a finite direct sum of objects in $\mathcal{A}$ and $\pi\in\mathrm{End}_{\mathcal{C}}(X)$ is idempotent. By the proof of \cite[Theorem 3.4]{Tuba-Wenzl},  we can choose the equivalence $\mathcal{C}\rightarrow{\rm Ab}\,\mathcal{A}$ to send 
$X(\omega_1)$ to $(X(\omega_1),\id_{X(\omega_1)})$ and $X(\Lambda)$ to $(X(\omega_1)^{\otimes 2}, \pi_\Lambda)$ for $\Lambda=2\omega_1,\omega_2,0$. 
We can also choose the third arrow ${\rm Ab}\,\widetilde{\mathcal{A}}\rightarrow\widetilde{\mathcal{C}}$ of $F$ to be a quasi-inverse of a tensor equivalence $\widetilde{\mathcal{C}}\rightarrow{\rm Ab}\,\widetilde{\mathcal{A}}$ defined in the same way on $\widetilde{X}(\Lambda)$ for $\Lambda=\omega_1,2\omega_1,\omega_2,0$. Then since $G(\id_{X(\omega_1)})=\id_{\widetilde{X}(\omega_1)}$ and $G(\pi_\Lambda)=\widetilde{\pi}_{\Lambda}$ for $\Lambda=2\omega_1,\omega_2,0$, these choices ensure that $F(X(\Lambda))\cong\widetilde{X}(\Lambda)$ for $\Lambda=0, \omega_1,\omega_2, 2\omega_1$.

    To show that $F(X(\Lambda))\cong\widetilde{X}(\Lambda)$ for all $\Lambda\in P^+$, we need some fusion rules from \cite[Section 6.1]{Tuba-Wenzl}. There, it is recalled that the simple objects of $\mathcal{C}$ and $\widetilde{\mathcal{C}}$ can be alternatively labeled by Young diagrams with at most $2n+1$ boxes in the first two columns, equivalently by integer partitions $P=(\lambda_1,\lambda_2,\ldots,\lambda_k)$ such that $\lambda_1\geq\lambda_2\geq\cdots\geq\lambda_k>0$ and $\lambda_1+\lambda_2\leq 2n+1$. If $X_P$ is the simple object of $\mathcal{C}$ labeled by $P$, then $X(\omega_1)\otimes X_P$ is the direct sum of $X_{P'}$ such that the Young diagram of $P'$ can be obtained from that of $P$ by adding or removing one box (and $\widetilde{\mathcal{C}}$ has the analogous fusion rule). The weights $0,\omega_1,\omega_2$, $2\omega_1$ correspond to Young diagrams with at most two boxes, so we already know that $F(X_P)\cong\widetilde{X}_P$ for partitions $P$ of $0$, $1$, and $2$. We now show that $F(X_P)\cong\widetilde{X}_P$ for partitions $P$ of any $N$ by induction on $N$. Thus assume that $P$ is a partition of $N\geq 3$ and that $F(X_{P'})\cong \widetilde{X}_{P'}$ for partitions $P'$ of any $N'<N$. 
    
    First assume that $P$ has at least two distinct parts, with the largest two parts occurring $p_1$ and $p_2$ times: $P=(\lambda_1^{p_1},\lambda_2^{p_2},\lambda_3,\ldots,\lambda_k)$ where $\lambda_1>\lambda_2>\lambda_3$. Then we define
     two distinct partitions of $N-1$:
    \begin{equation*}
P' =(\lambda_1^{p_1-1},\lambda_1-1,\lambda_2^{p_2},\lambda_3\ldots, \lambda_k),\qquad P''=(\lambda_1^{p_1},\lambda_2^{p_2-1},\lambda_2-1,\lambda_3,\ldots,\lambda_{k}).
    \end{equation*}
Both $X(\omega_1)\otimes X_{P'}$ and $X(\omega_1)\otimes X_{P''}$ contain $X_P$ as a direct summand, but there is no other partition $Q$ of $N$ such that both tensor products contain $X_Q$. Indeed, if $X_Q\subseteq X(\omega_1)\otimes X_{P'}$ and $Q\neq P$, then $\lambda_1$ is a part of $Q$ either $p_1-2$ or $p_1-1$ times, and in the second case $\lambda_1+1$ is not a part, while if $X_Q\subseteq X(\omega_1)\otimes X_{P''}$ and $Q\neq P$, then $\lambda_1$ is a part of $Q$ either $p_1+1$, $p_1$, or $p_1-1$ times, and in the last case, $\lambda_1+1$ is a part.
Now since $F$ is a tensor functor and $F(X_{P'})\cong\widetilde{X}_{P'}$ by induction, $F(X_P)$ is isomorphic to a simple direct summand $\widetilde{X}_Q$ of
\begin{equation*}
    F(X(\omega_1)\otimes X_{P'})\cong F(X(\omega_1))\otimes F(X_{P'})\cong\widetilde{X}(\omega_1)\otimes\widetilde{X}_{P'},
\end{equation*}
and $Q$ is a partition of $N$ since otherwise it would be a partition of $N-2$, but then $\widetilde{X}_Q$ would already be isomorphic to $F(X_Q)$ by induction. Similarly, $\widetilde{X}_Q$ must be a direct summand of $\widetilde{X}(\omega_1)\otimes\widetilde{X}_{P''}$, forcing $Q=P$. Thus $F(X_P)\cong\widetilde{X}_P$ when $P$ is a partition of $N$ with at least two distinct parts.

Now let $P=(\lambda^p)$ have one distinct part, and set $P'=(\lambda^{p-1},\lambda-1)$. Again since $F$ is a tensor functor and by induction, $F(X_P)$ is a direct summand of
\begin{equation*}
    \widetilde{X}(\omega_1)\otimes\widetilde{X}_{P'} \cong \left[\widetilde{X}_{(\lambda+1,\lambda^{p-2},\lambda-1)}\oplus\right] \widetilde{X}_P\oplus\widetilde{X}_{(\lambda^{p-1},\lambda-1,1)}\oplus\widetilde{Y},
\end{equation*}
where $\widetilde{Y}$ is a sum of $\widetilde{X}_Q$ for partitions $Q$ of $N-2$, and the summand in brackets does not occur if $p=1$. Now, both $(\lambda+1,\lambda^{p-2},\lambda-1)$ and $(\lambda^{p-1},\lambda-1,1)$ have two distinct parts unless $p=1$ and $\lambda=2$, in which case $N=2$. Since $N\geq 3$, the previous paragraph and induction imply that $F(X_P)$ cannot be isomorphic to $\widetilde{X}_{(\lambda+1,\lambda^{p-2},\lambda-1)}$, $\widetilde{X}_{(\lambda^{p-1},\lambda-1,1)}$, or any direct summand of $\widetilde{Y}$. Thus $F(X_P)\cong\widetilde{X}_P$, completing the inductive proof of $F(X(\Lambda))\cong\widetilde{X}(\Lambda)$ for all $\Lambda\in P^+$.

Finally, we need to show the tensor equivalence $F$ is braided. Since $X(\omega_1)$ tensor generates $\mathcal{C}$, it is enough to show that the isomorphism
\begin{align*}
    F(X(\omega_1))\otimes F(X(\omega_1))  \xrightarrow{\sim} & \,F(X(\omega_1)\otimes X(\omega_1))\\
    \xrightarrow{F(c_{\omega_1,\omega_1})} &\,F(X(\omega_1)\otimes X(\omega_1))\xrightarrow{\sim} F(X(\omega_1))\otimes F(X(\omega_1))
\end{align*}
(where the first and third arrows are from the data of the tensor functor $F$) is the self-braiding of $F(X(\omega_1))\otimes F(X(\omega_1))$ in $\widetilde{\mathcal{C}}$. In fact, under the identification
\begin{equation}\label{F-braid-1}
    F(X(\omega_1))\otimes F(X(\omega_1))\cong\widetilde{X}(\omega_1)\otimes\widetilde{X}(\omega_1)\cong\widetilde{X}(2\omega_1)\oplus\widetilde{X}(\omega_2)\oplus\widetilde{X}(0),
\end{equation}
this isomorphism is identified with
\begin{equation}\label{F-braid-2}
    a_{2\omega_1}\cdot\id_{\widetilde{X}(2\omega_1)}\oplus a_{\omega_2}\cdot\id_{\widetilde{X}(\omega_2)}\oplus a_{0}\cdot\id_{\widetilde{X}(0)},
\end{equation}
    where $(a_{2\omega_1},a_{\omega_2},a_0)$ is the braiding eigenvalue triple of $\mathcal{C}$. Since $\widetilde{\mathcal{C}}$ has the same braiding eigenvalue triple, and since the braiding on $\widetilde{\mathcal{C}}$ is natural, \eqref{F-braid-2} is indeed identified with the self-braiding of $F(X(\omega_1))\otimes F(X(\omega_1))$ under \eqref{F-braid-1}. This proves that $F$ is braided.
\end{proof}

Since we have determined the braiding eigenvalue triples of the representation categories of the affine Lie (super)algebras and quantum (super)groups of $\mathfrak{so}_{2n+1}$ and $\mathfrak{osp}_{1|2n}$, we can now obtain braided tensor equivalences between them.
First, comparing \eqref{quantum-so-braiding} and \eqref{braidingaffineso} and applying Theorem \ref{maintheoremtubawenzl}, we obtain the Kazhdan-Lusztig correspondence for non-spinorial $\mathfrak{so}_{2n+1}$ at irrational levels:
\begin{theorem}\label{thm:KLB}
    For $k\in\mathbb{C}\setminus\mathbb{Q}$ and $\zeta = e^{\pi i/2(k+2n-1)}$, there is a braided tensor equivalence $\KL_k^{\rm ns}(\mathfrak{so}_{2n+1})\rightarrow\mathcal{C}^{\rm ns}_\zeta(\mathfrak{so}_{2n+1})$ sending $N_k(\Lambda)$ to the finite-dimensional irreducible $U_\zeta(\mathfrak{so}_{2n+1})$-module of highest weight $\Lambda$.
\end{theorem}

\begin{rema}
Assuming the results of \cite{KL1,KL2,KL3,KL4} generalize to $\mathfrak{so}_{2n+1}$ at irrational levels, the entire categories $\KL_k(\mathfrak{so}_{2n+1})$ and $\mathcal{C}_\zeta(\mathfrak{so}_{2n+1})$ are also braided tensor equivalent for $\zeta=e^{\pi i/2(k+2n-1)}$.
\end{rema}

Next, comparing \eqref{braidingaffineosp} with the result of \cite[Lemma 7.2]{B} quoted in Section \ref{subsec:qg-braidings} and applying Theorem \ref{maintheoremtubawenzl}, 
we establish the Kazhdan-Lusztig correspondence for $\mathfrak{osp}_{1|2n}$ at irrational levels:
\begin{theorem}\label{thm:KLosp}
    For $k\in\mathbb{C}\setminus\mathbb{Q}$ and $q=e^{\pi i/(2k+2n+1)}$, there is a braided tensor equivalence 
     $\KL_k^{\rm ev}(\mathfrak{osp}_{1|2n})\rightarrow \mathcal{C}_q^{\rm ev}(\mathfrak{osp}_{1|2n})$ sending $M_k(\Lambda)$ to the even finite-dimensional $U_q(\mathfrak{osp}_{1\vert 2n})$-module of highest weight $\Lambda$. 
\end{theorem}

Finally, we observe that the tensor equivalence between the representation categories of quantum $\mathfrak{so}_{2n+1}$ and $\mathfrak{osp}_{1\vert 2n}$ established in \cite{Zhang:1992b} is braided:
\begin{theorem}\label{thm:QQ}
    If $\zeta\in\mathbb{C}^\times$ is not a root of unity, then there is a braided tensor equivalence $\mathcal{C}_{\zeta}^{\rm ns}({\mathfrak{so}_{2n+1}})\rightarrow\mathcal{C}_{-\zeta^2}^{\rm ev}({\mathfrak{osp}_{1|2n}})$ that matches (even) irreducible modules of highest weight $\Lambda$ on both sides.
\end{theorem}

As different values of $k$ in Theorems \ref{thm:KLB} and \ref{thm:KLosp} may yield the same values of $\zeta$ and $q$, we can summarize Theorems \ref{thm:KLB}, \ref{thm:KLosp}, and \ref{thm:QQ} as follows:

\begin{theorem}\label{thm:all}
If $k,\ell\in\mathbb{C}\setminus\mathbb{Q}$ are related by
\[
\frac{1}{2k+2n+1} = \frac{1}{\ell+ 2n-1} + 1 \ \ ({\rm mod}\  2\mathbb Z)
\]
and $q,\zeta\in\mathbb{C}^\times$ are defined by $q=e^{\pi i/(2k+2n+1)}$ and $\zeta=e^{\pi i/2(\ell+2n-1)}$, so $q=-\zeta^2$ in particular,
then there are braided tensor equivalences
\begin{equation}\label{equivalences}
\KL_k^{\rm ev}(\mathfrak{osp}_{1|2n}) \cong\mathcal{C}_{q}^{\rm ev}(\mathfrak{osp}_{1|2n}) \cong \mathcal{C}_{\zeta}^{\rm ns}({\mathfrak{so}_{2n+1}})\cong  \KL_\ell^{\rm ns}(\mathfrak{so}_{2n+1})
\end{equation}
which match objects labeled by $\Lambda\in P^+$ in each category.
\end{theorem}

\section{Mixed kernel VOAs}

In this section, we use the equivalences of the previous section to construct new simple conformal vertex (super)algebras whose existence was conjectured in \cite{CLNS} (and some results in \cite{CLNS} are conditional on their existence).
We first recall the gluing construction of VOAs from \cite{CKM2}.
If $\cC$ and $\cD$ are semisimple rigid braided tensor categories and $\tau: \cC \rightarrow \cD$ is a braid-reversed tensor equivalence, then the canonical algebra
\[
A = \bigoplus_{X \in {\rm Irr}\,\cC}X^* \boxtimes \tau(X)
\]
is a simple commutative algebra in (the ind-category of) the Deligne tensor product $\mathcal{C}\boxtimes\mathcal{D}$ (see \cite[Main Theorem~1]{CKM2}).
If $\cC$ and $\cD$ are representation categories of VOAs $U$ and $V$, then the vector space tensor product $\otimes$ induces a braided tensor equivalence from $\cC\boxtimes\cD$ to the category of $U\otimes V$-modules whose objects are finite direct sums of modules $M\otimes W$ where $M$ is in $\cC$ and $W$ is in $\cD$ \cite[Theorem 5.5]{CKM2}. If $\cC$ and $\cD$ have infinitely many simple objects, the direct limit completion ${\rm Ind}(\cC\boxtimes\cD)$ is also a vertex algebraic braided tensor category under mild conditions \cite[Theorem 1.1]{CMY1}, and thus the commutative algebra $A$ in ${\rm Ind}(\cC\boxtimes\cD)$ has the structure of a simple conformal vertex algebra extension of $U\otimes V$ \cite[Theorem 3.2]{HKL}. 

In this section, we need one or both of $U$ and $V$ to be $V_k(\mathfrak{osp}_{1\vert 2n})$, so we need the gluing construction to generalize to the case that one or both of $\cC$ and $\cD$ is a suitable module category of a vertex operator superalgebra. The result we get is the following:
\begin{theorem}\label{thm:superalgebra-gluing}
    Let $\cC$ and $\cD$ be locally finite semisimple braided tensor categories of grading-restricted modules for vertex operator superalgebras $U$ and $V$, respectively, and assume:
    \begin{enumerate}
        \item There are no non-zero odd $U$- or $V$-module homomorphisms between simple modules in $\cC$ or $\cD$.

        \item If $\mathcal{Y}$ is an intertwining operator of type $\binom{X}{M_1\otimes W_1\,\,M_2\otimes W_2}$ where $M_1$, $M_2$ are simple modules in $\cC$, $W_1$, $W_2$ are simple modules in $\cD$, and $X$ is a possibly infinite direct sum of modules $M\otimes W$ where $M$ and $W$ are simple objects of $\cC$ and $\cD$, respectively, then the image of $\mathcal{Y}$ in $X$ is a finite direct sum of simple $U\otimes V$-modules.


        \item There is a braid-reversed tensor equivalence $\tau:\cC\rightarrow\cD$.
    \end{enumerate}
    Then $A=\bigoplus_{X\in{\rm Irr}\,\cC} X^*\otimes\tau(X)$ has the structure of a simple conformal vertex superalgebra extending $U\otimes V$.
\end{theorem}

\begin{proof}
Here we just sketch how the results of \cite{CKM2,CMY1,HKL} generalize to superalgebras, particularly emphasizing the role of Assumption 1. Since $\cC$ and $\cD$ are locally finite abelian and semisimple, any of their objects is a finite direct sum of simple objects. The objects of $\cC$ and $\cD$ have $\mathbb{Z}/2\mathbb{Z}$-gradings, and all morphisms are even.
Define $\mathcal{E}$ to be the locally finite abelian and semisimple category of $U\otimes V$-modules consisting of finite direct sums of modules $M\otimes W$ where $M$ is simple in $\cC$ and $W$ is simple in $\cD$.  Here, $\otimes$ is the tensor product of vector superspaces; in particular, if $f$ and $g$ are parity-homogeneous linear maps on $M$ and $W$, then $f\otimes g$ is defined by
\begin{equation*}
    (f\otimes g)(m\otimes w) =(-1)^{p(g)p(m)} f(m)\otimes g(w)
\end{equation*}
for parity-homogeneous $m\in M$, $w\in W$.


We first generalize \cite[Theorem 5.5]{CKM2} to show that the abstract Deligne tensor product $\cC\boxtimes\cD$ is equivalent to $\mathcal{E}$. 
Indeed, by the universal property of Deligne tensor products, there is a functor $F:\cC\boxtimes\cD\rightarrow\mathcal{E}$ such that $F(M\boxtimes W)=M\otimes W$ for all $M$ in $\cC$ and $W$ in $\cD$. This functor is essentially surjective since it is linear and since $\cC\boxtimes\cD$ and $\mathcal{E}$ are semisimple, and to show that $F$ is fully faithful and thus an equivalence, it is enough to show that the linear map
\begin{align}\label{eqn:DTP-iso}
F:\mathrm{Hom}_\cC(M_1,M_2)\otimes\mathrm{Hom}_\cD(W_1,W_2) & \longrightarrow \mathrm{Hom}_{\mathcal{E}}(M_1\otimes W_1,M_2\otimes W_2)\\
f\otimes g & \longmapsto f\otimes g\nonumber
\end{align}
is an isomorphism for all simple $M_1,M_2$ in $\cC$ and $W_1,W_2$ in $\cD$. Here we have the tensor product of vector spaces on the left and of superspaces on the right, but this does not cause problems since all morphisms in $\cC$, $\cD$, and $\mathcal{E}$ are even. 
The proof that \eqref{eqn:DTP-iso} is an isomorphism is just a superalgebra generalization
of the proof of \cite[Proposition 2.10]{Mc}.
However, we need Assumption 1 in one step of the proof that \eqref{eqn:DTP-iso} is surjective, where for any even $U\otimes V$-homomorphism $f: M_1\otimes W_1\rightarrow M_2\otimes W_2$ and any parity-homogeneous $m_1\in M_1$ and $m_2'\in M_2'$ (where $M_2'$ is the graded dual of $M_2$), one considers
\begin{equation*}
(m_2'\otimes\id_{W_2})\circ f(m_1\otimes\bullet).    
\end{equation*}
This is a $V$-homomorphism of parity $p(m_1)+p(m_2')$ and thus is either $0$ or even, so it is in $\mathrm{Hom}_\cD(W_1,W_2)$. This helps show that $f$ is in the image of $F$, and it follows that $F: \mathcal{C}\boxtimes\mathcal{D}\rightarrow\mathcal{E}$ is an equivalence of categories under Assumption 1.



Next, we need to show that $\mathcal{E}$ admits vertex algebraic braided tensor structure such that $F$ is a braided tensor equivalence. Let $\boxtimes_U$ and $\boxtimes_V$ denote the tensor product operations on $\cC$ and $\cD$, and let $\cY_\boxtimes$ denote tensor product intertwining operators in either category. Then examining Theorem 5.3, Remark 5.4, and Theorem 5.5 in \cite{CKM2} and their proofs, and adding sign factors as appropriate, it is sufficient to show that the ordered pair
\begin{equation*}
    \left((M_1\boxtimes_U M_2)\otimes(W_1\boxtimes_V W_2), \mathcal{Y}_\boxtimes\otimes\mathcal{Y}_\boxtimes\right)
\end{equation*}
satisfies the universal property of vertex algebraic tensor products for any modules $M_1, M_2$ in $\cC$ and $W_1,W_2$ in $\cD$. (Again, $\otimes$ here represents tensor products in the category of vector superspaces.) Moreover, by the proof of \cite[Theorem 5.3]{CKM2}, the universal property holds if \cite[Theorem 2.10]{ADL} can be generalized to intertwining operators. That is, for any triples of modules $M_1$, $M_2$, $M_3$ in $\cC$ and $W_1$, $W_2$, $W_3$ in $\cD$, we need the linear map
\begin{align}\label{eqn:intw-op-iso}
I\binom{M_3}{M_1\,M_2}\otimes I\binom{W_3}{W_1\,W_2} & \longrightarrow I\binom{M_3\otimes W_3}{M_1\otimes W_1\,\,M_2\otimes W_2}\\
\mathcal{Y}_U\otimes\mathcal{Y}_V &\longmapsto \mathcal{Y}_U\otimes\mathcal{Y}_V\nonumber
\end{align}
to be an isomorphism, where $I(\bullet)$ represents the vector space of even intertwining operators of the indicated type.

The proof that \eqref{eqn:intw-op-iso} is an isomorphism requires Assumption 1, which by the universal property of vertex algebraic tensor products implies that there are no non-zero odd intertwining operators among modules in $\cC$ or $\cD$. 
 First, \eqref{eqn:intw-op-iso} is injective exactly as in the proof of \cite[Theorem 2.10]{ADL}. To show it is also surjective, suppose $\mathcal{Y}\in I\binom{M_3\otimes W_3}{M_1\otimes W_1\,\,M_2\otimes W_2}$; then one can show that for any parity-homogeneous $m_1\in M_1$, $m_2\in M_2$, and $m_3'\in M_3'$, 
\begin{align*}
    \mathcal{Y}_{m_1,m_2}^{m_3'}: W_1\otimes W_2 &\longrightarrow W_3\lbrace x\rbrace\\
    w_1\otimes w_2 &\longmapsto (-1)^{p(m_2)p(w_1)} (x^{-L_U(0)} m_3'\otimes\id_{W_3})\circ\\
    &\qquad\qquad\circ\mathcal{Y}(x^{L_U(0)}m_1\otimes w_1,x)(x^{L_U(0)}m_2\otimes w_2)
\end{align*}
is a $V$-module intertwining operator of type $\binom{W_3}{W_1\,W_2}$ and parity $p(m_1)+p(m_2)+p(m_3')$. This is the superalgebra generalization of a variant of \cite[Proposition 2.11]{ADL}. By Assumption 1, $\mathcal{Y}_{m_1,m_2}^{m_3'}$ is either $0$ or even, and in the second case $p(m_3')=p(m_1)+p(m_2)$. Thus we can write
\begin{equation*}
    \mathcal{Y}_{m_1,m_2}^{m_3'} =\sum_i c_i(m_3'\otimes m_1\otimes m_2)\cdot\mathcal{Y}_V^{(i)}
\end{equation*}
where the $\mathcal{Y}_V^{(i)}$ are basis elements of $I\binom{W_3}{W_1\,W_2}$ and $c_i: M_3'\otimes M_1\otimes M_2\rightarrow\mathbb{C}^{1\vert 0}$ are even linear maps.
The sum over $i$ is finite since $\cD$ is a locally finite category and thus $I\binom{W_3}{W_1\,W_2}\cong\mathrm{Hom}_\cD(W_1\boxtimes_V W_2,W_3)$ is finite dimensional.
We can then define an even linear map $\mathcal{Y}_U^{(i)}: M_1\otimes M_2\rightarrow M_3\lbrace x\rbrace$ by
\begin{equation*}
    \langle m_3',\mathcal{Y}_U^{(i)}(m_1,x)m_2\rangle = c_i(x^{L_U(0)}m_3'\otimes x^{-L_U(0)}m_1\otimes x^{-L_U(0)}m_2).
\end{equation*}
One can then use the definitions and the linear independence of the $\mathcal{Y}_V^{(i)}$ to show that each $\mathcal{Y}_U^{(i)}$ is an intertwining operator in $I\binom{M_3}{M_1\,M_2}$ and that $\mathcal{Y}=\sum_i\mathcal{Y}_U^{(i)}\otimes\mathcal{Y}_V^{(i)}$, where $\otimes$ is the tensor product in the category of superspaces thanks to the factor of $(-1)^{p(m_2)p(w_1)}$ in the definition of $\mathcal{Y}_U^{(i)}$. Thus \eqref{eqn:intw-op-iso} is an isomorphism as required. 

So far, we have shown that $\cC\boxtimes\cD$ is braided tensor equivalent to the vertex algebraic braided tensor category $\mathcal{E}$ of $U\otimes V$-modules 
Next, in case $\cC$ and $\cD$ have infinitely many simple modules, \cite[Theorem 1.1]{CMY1} easily generalizes to superalgebras (just replace vector spaces with vector superspaces and add sign factors as appropriate in \cite[Section 6]{CMY1}). So the direct limit completion ${\rm Ind}(\cC\boxtimes\cD)$
is also a vertex algebraic braided tensor category of $U\otimes V$-modules under the conditions of \cite[Theorem 1.1]{CMY1}. Because $\cC\boxtimes\cD$ is semisimple, ${\rm Ind}(\mathcal{\cC\boxtimes\cD})$ is just the category of arbitrary, possibly infinite, direct sums of $U\otimes V$-modules $M\otimes W$ for $M$ simple in $\cC$ and $W$ simple in $\cD$. Then the only non-trivial condition of \cite[Theorem 1.1]{CMY1} follows from Assumption 2, so ${\rm Ind}(\cC\boxtimes\cD)$ is a braided tensor category of $U\otimes V$-modules.



Finally, the proof of \cite[Theorem 3.2]{HKL} immediately generalizes to show that if $A$ is a commutative algebra in a braided tensor category of modules for a vertex operator superalgebra, then $A$ has a conformal vertex superalgebra structure (just add a sign factor in the proof that the commutativity of $A$ yields the super skew-symmetry property of a conformal vertex superalgebra). Thus given the braid-reversed equivalence $\tau$ of Assumption 3, the simple canonical algebra $A=\bigoplus_{X\in{\rm Irr}\,\mathcal{C}} X^*\otimes\tau(X)$ of \cite[Main Theorem 1]{CKM2} has the structure of a simple conformal vertex superalgebra extending $U\otimes V$.
\end{proof}

\begin{rema}
The map \eqref{eqn:DTP-iso} might not be an isomorphism if Assumption 1 does not hold. For example,
the parity involutions $p_M$ and $p_W$ of simple objects $M$ and $W$ in $\cC$ and $\cD$ define odd isomorphisms from $M$ and $W$ to their parity reversals $\Pi(M)$ and $\Pi(W)$, and then $p_M\otimes p_W: M\otimes W\rightarrow\Pi(M)\otimes\Pi(W)$ (the vector superspace tensor product) is an even isomorphism, although $\mathrm{Hom}_\cC(M,\Pi(M))$ and $\mathrm{Hom}_\cD(W,\Pi(W))$ are $0$ in general.
However, if Assumption 1 holds for $\cC$, then $M$ and $\Pi(M)$ cannot both be objects of $\cC$, and similarly for $\cD$.
\end{rema}

We now show that the assumptions of Theorem \ref{thm:superalgebra-gluing} hold for the categories of interest in this paper:

\begin{lemma}\label{lem:gluing-conditions}
    Assumptions 1 and 2 in Theorem \ref{thm:superalgebra-gluing} hold when $\cC$ and $\cD$ are either $\KL_k^{\rm ev}(\mathfrak{osp}_{1\vert 2n})$ or $\KL_\ell^{\rm ns}(\mathfrak{so}_{2n+1})$ for $k,\ell\in\mathbb{C}\setminus\mathbb{Q}$.
\end{lemma}
\begin{proof}
Assumption 1 clearly holds for $\KL_\ell^{\rm ns}(\mathfrak{so}_{2n+1})$ (view $V_\ell(\mathfrak{so}_{2n+1})$ as a vertex operator superalgebra with zero odd part, so that objects of $\KL_\ell^{\rm ns}(\mathfrak{so}_{2n+1})$ can be $\mathbb{Z}/2\mathbb{Z}$-graded with zero odd part). For $\KL_k^{\rm ev}(\mathfrak{osp}_{1\vert 2n})$, any even or odd non-zero $V_k(\mathfrak{osp}_{1\vert 2n})$-homomorphism $M_k(\Lambda_1)\rightarrow M_k(\Lambda_2)$ commutes with the Cartan subalgebra of $\mathfrak{osp}_{1\vert 2n}$ and preserves highest-weight vectors, forcing $\Lambda_1=\Lambda_2$, so that we get an (even) non-zero multiple of the identity.

    For Assumption 2, let $\mathcal{Y}$ be an intertwining operator of type $\binom{X}{X_1\,X_2}$ where $X_1$ and $X_2$ are simple objects of $\cC\boxtimes\cD$. Then $X_1$ and $X_2$ are
    $C_1$-cofinite $U\otimes V$-modules for $U,V\in\lbrace V_k(\mathfrak{osp}_{1\vert 2n}), V_\ell(\mathfrak{so}_{2n+1})\rbrace$, so the image of $\mathcal{Y}$ is $C_1$-cofinite by \cite[Key Theorem]{Miy} and hence can only be a finite (rather than infinite) direct sum of simple $U\otimes V$-modules.
\end{proof}

Now we apply the above results to construct new conformal vertex algebras from braid-reversed tensor equivalences between the categories $\KL_k^{\rm ev}(\mathfrak{osp}_{1\vert 2n})$ and $\KL_\ell^{\rm ns}(\mathfrak{so}_{2n+1})$ for $k,\ell\in\mathbb{C}\setminus\mathbb{Q}$. In general, if $\cC$ is a category of non-spinorial $\mathfrak{so}_{2n+1}$ type with braiding eigenvalue triple $(q,-q^{-1},\pm q^{-2n})$, then the braid-reversed category $\cC^{\rm rev}$ is still of non-spinorial $\mathfrak{so}_{2n+1}$ type, with braiding eigenvalue triple $(q^{-1},-q,\pm q^{2n}$). Thus from \eqref{equivalences}, we also get braided equivalences
\begin{equation*}
\begin{split}
& \KL_k^{\rm ev}(\mathfrak{osp}_{1|2n})^{\rm rev} \cong \KL_{k'}^{\rm ev}(\mathfrak{osp}_{1|2n}), \\
& \KL_k^{\rm ev}(\mathfrak{osp}_{1|2n})^{\rm rev} \cong \KL_\ell^{\rm ns}(\mathfrak{so}_{2n+1}),\\
&\KL_\ell^{\rm ns}(\mathfrak{so}_{2n+1})^{\rm rev}\cong \KL_{\ell'}^{\rm ns}(\mathfrak{so}_{2n+1})
\end{split}
\end{equation*}
which match simple objects labeled by $\Lambda\in P^+$ for
\begin{equation*}
    \begin{split}
    \frac{1}{2k+2n+1} + \frac{1}{2k'+2n+1} & = 0  \mod 2\mathbb Z,\\
         \frac{1}{2k+2n+1} + \frac{1}{\ell+2n-1} &= 1  \mod 2\mathbb Z, \\
          \frac{1}{\ell+2n-1} + \frac{1}{\ell'+2n-1} &= 0  \mod 2\mathbb Z, \\
    \end{split}
\end{equation*}
respectively. Now Theorem \ref{thm:superalgebra-gluing} and Lemma \ref{lem:gluing-conditions} imply:
\begin{theorem}\label{thm:mixed-kernel-VOAs}
    The following are simple conformal vertex (super)algebras extending $V_k(\mathfrak{osp}_{1|2n})\otimes V_{k'}(\mathfrak{osp}_{1|2n})$, $V_k(\mathfrak{osp}_{1|2n})\otimes V_{\ell}(\mathfrak{so}_{2n+1})$, and $V_\ell(\mathfrak{so}_{2n+1})\otimes V_{\ell'}(\mathfrak{so}_{2n+1})$, respectively:
    \begin{enumerate}
        \item For $k\in\mathbb{C}\setminus\mathbb{Q}$ and $m\in2\mathbb{Z}$,
        \begin{equation*}
         A[\mathfrak{osp}_{1|2n}, \mathfrak{osp}_{1|2n}, m, k] = \bigoplus_{\Lambda \in P^+} M_k(\Lambda) \otimes M_{k'}(\Lambda)   
        \end{equation*}
        where $k'\in\mathbb{C}\setminus\mathbb{Q}$ is defined by $\frac{1}{2k+2n+1} + \frac{1}{2k'+2n+1} = m$.

        \item For $k\in\mathbb{C}\setminus\mathbb{Q}$ and $p\in2\mathbb{Z}+1$,
        \begin{equation*}
         A[\mathfrak{osp}_{1|2n}, \mathfrak{so}_{2n+1}, p, k] = \bigoplus_{\Lambda \in P^+} M_k(\Lambda) \otimes N_{\ell}(\Lambda) 
        \end{equation*}
        where $\ell\in\mathbb{C}\setminus\mathbb{Q}$ is defined by $\frac{1}{2k+2n+1} + \frac{1}{\ell+2n-1} = p$.

        \item For $\ell\in\mathbb{C}\setminus\mathbb{Q}$ and $m\in2\mathbb{Z}$,
        \begin{equation*}
         A[\mathfrak{so}_{2n+1}, \mathfrak{so}_{2n+1}, m, \ell]^{\rm ns} = \bigoplus_{\Lambda \in P^+} N_\ell(\Lambda) \otimes N_{\ell'}(\Lambda)  
        \end{equation*}
        where $\ell'\in\mathbb{C}\setminus\mathbb{Q}$ is defined by $\frac{1}{\ell+2n-1} + \frac{1}{\ell'+2n-1} = m$.
    \end{enumerate}
    Moreover, $A[\mathfrak{osp}_{1|2n}, \mathfrak{osp}_{1|2n}, m, k]$ and $A[\mathfrak{so}_{2n+1}, \mathfrak{so}_{2n+1}, m, \ell]^{\rm ns}$ are $\mathbb{Z}$-graded, while $A[\mathfrak{osp}_{1|2n}, \mathfrak{so}_{2n+1}, p, k]$ is $\frac{1}{2}\mathbb{Z}$-graded.
\end{theorem}

\begin{proof}
    The grading assertions follow because the definitions of $h(M_k(\Lambda))$ and $h(N_k(\Lambda))$ in Section \ref{aff} imply that the lowest conformal weight of the summand labeled by $\Lambda$ in parts 1, 2, and 3 of the theorem is $\frac{m}{2}(\Lambda\vert\Lambda+2\rho)$, $\frac{p}{2}(\Lambda\vert\Lambda+2\rho)$, and $\frac{m}{2}(\Lambda\vert\Lambda+2\rho)$, respectively, where $(\Lambda\vert\Lambda+2\rho)\in\mathbb{Z}$.
\end{proof}

        

\end{document}